\documentclass[12pt, reqno]{amsart}
\usepackage{amsmath, amsthm, amscd, amsfonts, amssymb}
\usepackage{bm}
\usepackage[all,cmtip]{xy}
\usepackage{xcolor}
\newtheorem{theorem}{Theorem}[section]
\newtheorem{lemma}[theorem]{Lemma}
\newtheorem{proposition}[theorem]{Proposition}
\newtheorem{corollary}[theorem]{Corollary}
\theoremstyle{definition}

\newtheorem{example}[theorem]{Example}

\newtheorem{question}[theorem]{Question}

\numberwithin{equation}{section}

\newcommand{\vp}{\varphi}

\newcommand{\clb}{\mathcal{B}}

\newcommand{\clh}{\mathcal{H}}
\newcommand{\clk}{\mathcal{K}}

\newcommand{\clq}{\mathcal{Q}}

\newcommand{\D}{\mathbb{D}}

\newcommand{\T}{\mathbb{T}}
\newcommand{\Z}{\mathbb{Z}}
\newcommand{\raro}{\rightarrow}

\begin{document}

\title[Liftings and invariant subspaces]{Liftings and invariant subspaces of Hankel operators}


\author[Sneha B]{Sneha B}
\address{Statistics and Mathematics Unit, Indian Statistical Institute, 8th Mile, Mysore Road, Bangalore, 560059, India}
\email{sneharbkrishnan@gmail.com, rs\textunderscore math2105@isibang.ac.in}

\author[Bala]{Neeru Bala}
\address{IIT(ISM) Dhanbad, Jharkhand, 826004, India}
\email{neerusingh41@gmail.com, neerubala@iitism.ac.in}

\author[Panja]{Samir Panja}
\address{Statistics and Mathematics Unit, Indian Statistical Institute, 8th Mile, Mysore Road, Bangalore, 560059, India}
\email{samirpanja\textunderscore pd@isibang.ac.in, panjasamir2020@gmail.com}

\author[Sarkar]{Jaydeb Sarkar}
\address{Statistics and Mathematics Unit, Indian Statistical Institute, 8th Mile, Mysore Road, Bangalore, 560059, India}
\email{jay@isibang.ac.in, jaydeb@gmail.com}

\subjclass[2020]{30H10, 47B35, 32A35, 47B38, 46E20, 15B05}
\keywords{Hardy space, Hankel operators, model spaces, Hilbert matrix, invariant subspaces, reducing subspaces}

\begin{abstract}
We prove a Hankel-variant commutant lifting theorem. This also uncovers the complete structure of the Beurling-type reducing and invariant subspaces of Hankel operators. Kernel spaces of Hankel operators play a key role in the analysis.
\end{abstract}

\maketitle

\tableofcontents

\section{Introduction}

Sarason's commutant lifting theorem \cite{Sara} is a cornerstone of Hilbert function spaces and has many applications in various fields. This result is powerful as it also stands on a core philosophy. It says that one may lift an operator commuting with a model operator to an operator commuting with the shift operator without altering the norm. Model operators are simply compressions of the shift operator into its co-invariant subspaces.

Now, the commutant of the shift is analytic Toeplitz operators, whereas Hankel operators act as intertwiners between the shift and its adjoint. Then, from the perspective of Sarason's lifting theorem or just general interest, a natural question arises as to whether the intertwiners between a model operator and its adjoint lift to Hankel operators.

Curiously, this question has gone unnoticed, and the purpose of this paper is to address it. On one hand, our approach to proving this lifting theorem is somewhat simpler. On the other hand, even at a very basic level, unlike Sarason's lifting, we observe that the existence of nonzero intertwiners between a model operator and its adjoint is not guaranteed. We point out that the underlying structure of model spaces, which reduces Hankel operators, conceals the existential problem. We solve this reducing (and even invariant) subspace problem and use it to solve the existential problem. Our classification provides concrete examples of both the possible and impossible cases of intertwiners.

Now we delve further into the subject matter of the paper. Let $L^2(\T)$ represent the classical Hilbert space of Lebesgue square integrable functions on $\T = \{z \in \mathbb{C}: |z|=1\}$. Let $L_\vp$ denote the Laurant operator on $L^2(\T)$ that corresponds to the symbol $\vp \in L^\infty(\T)$. Here, $L^\infty(\T)$ is the standard von Neumann algebra of $\mathbb{C}$-valued essentially bounded Lebesgue measurable functions on $\T$. Therefore, we have
\[
L_\vp f = \vp f,
\]
for all $f \in L^2(\T)$. This is a repository for many profound theories that encompass both Toeplitz operators and Hankel operators. The \textit{Toeplitz operator} $T_\vp$ and the \textit{Hankel operator} $H_\vp$ with the symbol $\vp \in L^\infty(\T)$ are defined as
\[
T_\vp = P_+ L_\vp|_{H^2},
\]
and
\[
H_\varphi = P_{+} L_\vp J|_{H^2},
\]
respectively. By $H^2$ we mean the Hardy space of square integrable functions on $\T$ and analytic on $\D = \{z \in \mathbb{C}: |z| < 1\}$. Alternatively, this is the space of all $L^2(\T)$-functions that have zero negative Fourier coefficients. Also, $P_{+}$ denotes the orthogonal projection from $L^2 (\mathbb{T})$ onto $H^2$, and $J: L^2 (\mathbb{T}) \to L^2 (\mathbb{T})$ is the unitary operator defined by
\[
(Jf)(z)=f(\bar{z}),
\]
for all $f\in L^2 (\mathbb{T})$ and $z \in \T$. Note that the special Toeplitz operator $T_z$ is an isometry and is also popularly known as the \textit{shift operator}. By application of the Beurling theorem \cite{Ber}, we know that $T_z^*$-invariant closed subspaces are represented as
\[
\clq_u = H^2\ominus u H^2,
\]
where $u \in H^\infty(\D)$ is an inner function. This is referred to as a \textit{model space}, and the \textit{model operator} is the compression
\[
S_u = P_{\clq_u} T_z|_{\clq_u},
\]
where $P_{\clq_u}$ denotes the orthogonal projection of $H^2$ onto $\clq_u$. Recall that $H^\infty(\D)$ is the Banach algebra of all bounded analytic functions on $\D$, and a function $u \in H^\infty(\D)$ is \textit{inner} if
\[
|u(z)| = 1,
\]
for a.e. $z \in \T$ (in the sense of radial limits). We refer to $uH^2$ as a \textit{Beurling-type subspace} of $H^2$. The commutant of $T_z$ is precisely the Toeplitz operators with symbols from $H^\infty(\D)$. In other words
\[
\{X \in \clb(H^2): T_z X = X T_z\} = \{T_\vp: \vp \in H^\infty(\D)\}.
\]
Throughout the paper, $\clb(\clh)$ refers to the space of all bounded linear operators on a given Hilbert space $\clh$.

It is now desirable to recall Sarason's commutant lifting theorem \cite{Sara}. Let $X \in \clb(\clq_u)$. Then
\[
S_u X = X S_u,
\]
if and only if there exists a function $\vp \in H^\infty(\D)$ such that $\|X\| = \|T_\vp\|$ and
\[
X = P_{\clq_u} T_\vp|_{\clq_u}.
\]
This amounts to the commutativity of the following diagram (when combined with the norm identity $\|X\| = \|T_\vp\|$):
\[
\xymatrix{
H^2 \ar@{->}[rr]^{\displaystyle T_{\vp}} \ar@{<-}[dd]_{\displaystyle i_{\clq_u}}
&& H^2 \ar@{->}[dd]^{\displaystyle P_{\clq_u}}    \\ \\
\clq_u \ar@{->}[rr]_{\displaystyle X} && \clq_u
}
\]
where $i_{\clq_u}: \clq_u \hookrightarrow H^2$ is the canonical isometric inclusion operator.

With this perspective, we now turn to Hankel operators. Following the commutant of $T_z$ described above, we first recall the set of intertwiners between the shift $T_z$ and its adjoint $T_z^*$:
\begin{equation}\label{eqn: Hankel class}
\{X \in \clb(H^2): T_z^* X = X T_z\} = \{H_\vp: \vp \in L^\infty(\T)\}.
\end{equation}
This well-known result is due to Nehari \cite{Neha} (also see \cite[Theorem 2.2.4]{Nik TO} for a proof). When it comes to Hankel operators, this is one of the most important and useful results. This, together with Sarason's lifting theorem, raises the following natural question:

\begin{question}\label{Quest 1}
Given an inner function $u \in H^\infty(\D)$ and an operator $X\in\clb(\clq_u)$, determine whether the condition
\begin{align}\label{Main_id}
S_u^* X= X S_u,
\end{align}
implies that $X$ lifts to a Hankel operator maintaining $X$'s norm. That is, whether there exists a symbol $\varphi\in L^{\infty}(\mathbb{T})$ for which that
\[
X = P_{\clq_u} H_\vp|_{\clq_u},
\]
and
\[
\|X\| = \|H_\vp\|.
\]
\end{question}

Of course, this question is so basic that it arises as a matter of independent interest. The preceding diagram can be used to illustrate the problem in a manner similar to the commutant lifting theorem. However, in this case, $X\in \clb(\clq_u)$ fulfills the intertwining relationship \eqref{Main_id}. In Theorem \ref{thm: lift 1}, we solve the lifting problem for Hankel operators:

\begin{theorem}\label{thm: intro 1}
Let $u \in H^\infty(\D)$ be an inner function and let $X\in\clb(\clq_u)$. Then $S_u^*X=XS_u$ if and only if there exists $\varphi\in L^{\infty}(\mathbb{T})$ such that
\[
H_{\varphi}=\begin{bmatrix}
X&0\\
0&0
\end{bmatrix}.
\]
In particular, if $S_u^*X=XS_u$, then $X=H_{\varphi}|_{\clq_u}$, and $\|H_\varphi\|=\|X\|$.
\end{theorem}

The proof of this theorem relies solely on two lemmas, namely Lemma \ref{prop1} and Lemma \ref{coro1}. The former lemma further uses Nehari's classifications of intertwiners between the shift $T_z$ and its adjoint $T_z^*$, as pointed out in \eqref{eqn: Hankel class}.

Theorem \ref{thm: intro 1} appears to be the solution to the lifting problem proposed in Question \ref{Quest 1} above. However, some thoughts prompt questions about even the existence of a nonzero $X \in \clb(\clq_u)$ satisfying the intertwining condition \eqref{Main_id}. In fact, there is an abundance of nonzero $X$ that fulfills $S_u X = X S_u$ in the case of the commutant lifting theorem (just take $X = S_u$). On the other hand, there isn't an automated fix for this in the Hankel operator situation. That is, given an inner function $u \in H^\infty(\D)$, it is not effortless to identify a nonzero solution $X \in\clb(\clq_u)$ to the identity \eqref{Main_id}. This prompts the question that follows:

\begin{question}
Classify nonconstant inner functions $u \in H^\infty(\D)$ for which there exists a nonzero operator $X \in \clb(\clq_u)$ satisfying the intertwining condition \eqref{Main_id}.
\end{question}

This question is fascinating on its own because it also involves solving operator equations. Theorem \ref{thm: nonzero lift} yields a complete answer to this question:

\begin{theorem}\label{thm: nonzero lift}
There exists a nonzero operator $X\in\clb(\clq_u)$ satisfying $S_u^*X=XS_u$ if and only if
\[
\gcd\{u, \overline{Ju}\}\ne 1.
\]
\end{theorem}

It is relevant to note in the above scenario that
\[
X = H_\vp|_{\clq_u},
\]
provides a nonzero solution for the operator equation $S_u^*X=XS_u$ (see the proof of Theorem \ref{thm: nonzero lift}, and specifically the construction in \eqref{eqn: nontrivial phi}), where
\[
\vp = T_z^* \gcd\{u, \overline{Ju}\}.
\]

The above result is a byproduct of the solution to another natural question about the structure of Beurling-type subspaces, which reduces Hankel operators. Indeed, a closer inspection of Theorem \ref{thm: intro 1} shows that the restrictions of Hankel operators on model spaces that reduce the corresponding Hankel operators provide the possibility of a nonzero solution $X$ satisfying \eqref{Main_id}. In this context, we first classify Beurling-type subspaces that are invariant under Hankel operators:

\begin{theorem}\label{invariant}
Let $u \in H^\infty(\D)$ be a nonconstant inner function, and let $\vp \in L^\infty(\T)$. Then, the following are equivalent:
\begin{enumerate}
\item $u H^2$ is invariant under $H_\varphi$.
\item $u H^2 \subseteq \ker H_\varphi$.
\item $P_+\varphi \in \clq_{\overline{Ju}}$.
\end{enumerate}
\end{theorem}

As a result, in Corollary \ref{reducing}, we present Beurling-type subspaces that reduce Hankel operators: In the setting of Theorem \ref{invariant}, the following are equivalent:
\begin{enumerate}
\item $u H^2$ reduces $H_\varphi$.
\item $u H^2 \subseteq \ker H_\varphi \cap \ker H^*_\varphi$.
\item $P_+\varphi\in \big(\text{gcd} \{u, \overline{Ju}\} H^2\big)^\perp$.
\end{enumerate}

Note that given a function $u = \sum_{n=0}^{\infty} a_n z^n \in H^\infty(\D)$, the function $\overline{Ju} \in H^\infty(\D)$ is given by
\[
\overline{Ju} = \sum_{n=0}^{\infty} \bar{a}_n z^n.
\]

We present examples and counterexamples of lifting, reducing, and invariant subspaces of Hankel operators. For instance, Theorem \ref{thm: nonzero lift} can be applied to a specific model space to verify the presence of nonzero intertwiners. We also have some properties of classical Hankel operators, like the one related to Hilbert's Hankel matrix $\Gamma$, where
\[
\Gamma = \Big\{\frac{1}{m+n+1}\Big\}_{m,n=0}^{\infty}.
\]
In particular, we have the following new information about the Hankel operator $H_\psi$, which corresponds to Hilbert's Hankel matrix:

\begin{enumerate}
\item $H_\psi$ cannot be viewed as a lift to any intertwiners of the model operator and its adjoint (see Example \ref{exam: Hilbert matrix} for more details).
\item Since $H_\psi$ is injective \cite{Mag}, by Theorem \ref{invariant}, it follows that there is no Beurling-type subspace that is invariant under $H_\gamma$. 
\end{enumerate}

We note that the theory of Hankel operators is an evergreen subject that is well-known for its connections to a variety of subjects and for solving problems of general interest. For instance, see Pisier's solution to the similarity problem \cite{Pisier} (also see \cite{AP}). While the monographs \cite{Nik TO, Peller} are excellent source for comprehensive discussion, we suggest reading \cite{Wolf, FTZ, GP, LR, Xia} to get a feel for the work from various perspectives.
Finally, we comment that the paper is self-contained modulo standard classical notions that can be found in the well-known monographs \cite{Foias, Nik TO, Peller, NF}.

The remainder of the paper is divided into four sections. The lifting theorem is presented in the next section, Section \ref{sec: lifting thm}. The lifting result of this section led us to discuss the Beurling-type invariant and reducing subspaces of Hankel operators, which are addressed in Section \ref{sec: red sub}. With these tools in hand, we revisit the lifting problem in Section \ref{sec: revisit lift} and refine the results obtained in Section \ref{sec: lifting thm}. The final section, Section \ref{sec: remark}, discusses possible generalizations and some refinements of the techniques applied in this paper.

\section{A lifting theorem}\label{sec: lifting thm}

The aim of this section is to prove the lifting theorem for operators acting on model spaces and satisfying the intertwining relation between the model operator and its adjoint.

In order to guarantee a proof that is technically rigorous, we shall make modifications to our notation. These alterations will solely be implemented in the proof of the subsequent two results (and also in a part of Section \ref{sec: remark}).


 Consider the canonical embedding
\[
i_{\clq_u} : \clq_u \hookrightarrow H^2.
\]
It is clear that
\[
i_{\clq_u}^* i_{\clq_u} = I_{\clq_u},
\]
and
\[
i_{\clq_u} i_{\clq_u}^* = P_{\clq_u},
\]
where, as usual, $P_{\clq_u} \in \clb(H^2)$ denotes the orthogonal projection of $H^2$ onto $\clq_u$. Using this notation, we can more appropriately define the model operator $S_u$ on the model space $\clq_u$ as follows:
\[
S_u = i_{\clq_u}^* T_z i_{\clq_u}.
\]
Evidently, this is the correct representation of the model operator, instead of following the standard definition of $S_u$ as $S_u = P_{\clq_u} T_z|_{\clq_u}$. Note also that $S_u^* = T_z^*|_{\clq_u}$. The following is a key to our lifting theorem:

\begin{lemma}\label{prop1}
Let $X \in \clb(\clq_u)$. Then $S_u^*X=XS_u$ if and only if there exists $\varphi\in L^{\infty}(\mathbb{T})$ such that
\[
i_{\clq_u} X i_{\clq_u}^* =H_{\varphi}.
\]
\end{lemma}
\begin{proof}
Suppose $S_u^*X=XS_u$. We claim that
\[
T_z^* (i_{\clq_u} X i_{\clq_u}^*) = (i_{\clq_u} X i_{\clq_u}^*) T_z.
\]
Indeed, $S_u^*X=XS_u$ implies that $i_{\clq_u}^* (T_z^* i_{\clq_u}) X = X (i_{\clq_u}^* T_z i_{\clq_u})$. Since $i_{\clq_u} i_{\clq_u}^* = P_{\clq_u}$, we have
\[
P_{\clq_u} T_z^* i_{\clq_u} X = (i_{\clq_u} X i_{\clq_u}^*) T_z i_{\clq_u}.
\]
But since
\begin{equation}\label{eqn: PQ Tz}
P_{\clq_u} T_z^* i_{\clq_u} = T_z^* i_{\clq_u},
\end{equation}
(as $T_z^* \clq_u \subseteq \clq_u$) we must have
\[
T_z^* i_{\clq_u} X = (i_{\clq_u} X i_{\clq_u}^*) T_z i_{\clq_u}.
\]
Multiplying both sides by $i_{\clq_u}^*$ on the right this implies
\[
T_z^* (i_{\clq_u} X i_{\clq_u}^*) = (i_{\clq_u} X i_{\clq_u}^*) T_z P_{\clq_u},
\]
which, in view of the fact that $i_{\clq_u}^* T_z P_{\clq_u} = i_{\clq_u}^* T_z$ (see \eqref{eqn: PQ Tz} above), finally yields the claim. Therefore, by \eqref{eqn: Hankel class}, there exists $\varphi\in L^{\infty}(\mathbb{T})$ such that $i_{\clq_u} X i_{\clq_u}^* = H_{\varphi}$.

\noindent For the converse direction, assume that there exists $\varphi\in L^{\infty}(\mathbb{T})$ such that $i_{\clq_u} X i_{\clq_u}^* = H_{\varphi}$. We compute
\[
\begin{split}
i_{\clq_u}(S_u^* X - X S_u) i_{\clq_u}^* & = ((i_{\clq_u} i_{\clq_u}^*) T_z^* i_{\clq_u}) (X i_{\clq_u}^*) - i_{\clq_u} X (i_{\clq_u}^* T_z (i_{\clq_u} i_{\clq_u}^*))
\\
& = P_{\clq_u} T_z^* i_{\clq_u} (X i_{\clq_u}^*) - i_{\clq_u} X (i_{\clq_u}^* T_z P_{\clq_u})
\\
& = T_z^* i_{\clq_u} X i_{\clq_u}^* - i_{\clq_u} X i_{\clq_u}^* T_z
\\
& = T_z^* H_\vp - H_\vp T_z
\\
& = 0,
\end{split}
\]
as $H_\vp$ is a Hankel operator (see \eqref{eqn: Hankel class}). Using the fact that $i_{\clq_u}$ is an isometry, we ultimately arrive at the conclusion that $S_u^* X = X S_u$. This completes the proof of the lemma.
\end{proof}

We also have the following elementary observation, which will be used as a tool in the proof of our lifting theorem.

\begin{lemma}\label{coro1}
Let $X \in \clb(\clq_u)$ and let $\vp \in L^\infty(\T)$. Then
\[
i_{\clq_u} X i_{\clq_u}^* =H_{\varphi},
\]
if and only if
\[
H_{\varphi}=\begin{bmatrix}
X&0\\
0&0
\end{bmatrix},
\]
with respect to the orthogonal decomposition $H^2=\clq_u\oplus\clq_u^{\perp}$.
\end{lemma}
\begin{proof}
We write $i_{\clq_u}: \clq_u \raro H^2 =\clq_u\oplus\clq_u^{\perp}$ as the column operator
\[
i_{\clq_u} = \begin{bmatrix}
I_{\clq_u}
\\
0
\end{bmatrix}.
\]
Then
\[
i_{\clq_u} X i_{\clq_u}^* = \begin{bmatrix}
I_{\clq_u}
\\
0
\end{bmatrix}
X \begin{bmatrix}
I_{\clq_u} & 0
\end{bmatrix},
\]
and hence
\begin{equation}\label{eqn: iQ X}
i_{\clq_u} X i_{\clq_u}^* = \begin{bmatrix}
X&0\\
0&0
\end{bmatrix}.
\end{equation}
The lemma is now easily derived from this identity.
\end{proof}

From the matrix representation of $H_\varphi$ in the above statement, it is evident that $\|H_\varphi\|=\|X\|$. As a result, Lemmas \ref{prop1} and \ref{coro1} immediately yield the lifting theorem for Hankel operators, as also stated in Theorem \ref{thm: intro 1} earlier.
 
 

\begin{theorem}\label{thm: lift 1}
Let $\clq_u$ be a model space and let $X\in\clb(\clq_u)$. Then $S_u^*X=XS_u$ if and only if there exists $\varphi\in L^{\infty}(\mathbb{T})$ such that
\[
H_{\varphi}=\begin{bmatrix}
X&0\\
0&0
\end{bmatrix}.
\]
In particular, if $S_u^*X=XS_u$, then we have the following:
\[
X=H_{\varphi}|_{\clq_u},
\]
and
\[
\|H_\varphi\|=\|X\|.
\]
\end{theorem}

This seemingly resolves the lifting theorem within the context of Hankel operators, providing a more direct method compared to the commutant lifting theorem \cite{Sara}. Nevertheless, this approach may obscure the underlying issue that does not persist in the commutant lifting theorem. We are essentially talking about the existence of nonzero intertwiners. In the case of commutant lifting, there is abundance of operators acting on model space that commute with the model operator. However, the question of the existence of nonzero operators that act on a model space and intertwine between the model operator and its adjoint does not appear to have a direct solution.

This suggests the above lifting result is incomplete, and in the following sections, we fill this gap. For instance, in Section \ref{sec: revisit lift}, we will present quantitative classifications of nonzero intertwiners.

\section{Invariant and reducing subspaces}\label{sec: red sub}

Theorem \ref{thm: lift 1} makes it clear that one needs to understand the structures of model spaces that reduce a given Hankel operator. The purpose of this section is to address that issue. In fact, we go a bit further: we first classify Beurling-type invariant subspaces of Hankel operators. As a consequence, this result would imply a classification of model spaces that reduces Hankel operators. Evidently, the invariant subspace theorem that follows this section is of independent interest.

We proceed to the proof of Theorem \ref{invariant}. Recall the unitary operator $J: H^2 \raro \overline{H^2}$ defined by $(Jf)(z) = f(\bar{z})$ for all $f \in H^2$ and $z \in \D$.

\bigskip

\noindent\textsf{Proof of Theorem \ref{invariant}:} To begin, we show that (1) $\Leftrightarrow$ (2). Assume that $H_\vp (u H^2) \subseteq u H^2$. Define a closed subspace $\clq \subseteq H^2$ as
\[
\clq = \overline{\text{span}} \{H_\vp (uf): f \in H^2\}.
\]
Clearly, $\clq \subseteq u H^2$. We claim that $\clq$ is a model space, that is, $\clq$ is $T_z^*$-invariant. To prove this claim, fix a function $g\in \clq$. Then, for every $n \in \mathbb{N}$, there exists a function $f_n\in H^2$ such that
\[
\|H_\varphi(uf_n)-g\|<\frac{1}{n}.
\]
Now, for any $f \in H^2$, we have
\[
\begin{split}
T_z^* (H_\varphi(uf) - g) & = T_z^*H_\varphi(uf) - T_z^*g
\\
& = H_\varphi T_z(uf) - T_z^* g
\\
& = H_\varphi(uzf) - T_z^*g.
\end{split}
\]
In the above, we have applied the Hankel property that $T_z^*H_\varphi = H_\vp T_z$. Since $T^*_z$ is a contraction, we have
\[
\|T_z^* H_\varphi(uf) - T_z^*g\|\leq \|H_\varphi(uf)-g\|.
\]
Thus, for all $n \in \mathbb{N}$, we have
\[
\|H_\varphi(uzf_n) - T_z^*g\| \leq \|H_\varphi(uf_n)-g\| <\frac{1}{n},
\]
which proves the claim that $\clq$ is a model space. Consequently, there is an inner function $v \in H^\infty(\D)$ such that
\[
\clq = \clq_v.
\]
Now we find ourselves in a situation where a model space is included in a Beurling-type invariant subspace, that is
\[
\clq_v \subseteq u H^2.
\]
We claim that this leads to triviality. Indeed, the above inclusion implies $z^n\clq_v \subseteq u H^2$ for all $n \in \Z_+$, and hence
\[
{\overline{\text{span}}} \{z^n \clq_v: n \in \Z_+\} \subseteq u H^2.
\]
But then, since ${\overline{\text{span}}} \{z^n \clq_v: n \in \Z_+\}$ reduces $T_z$ (recall that $T_z$ is an irreducible operator), we have
\[
{\overline{\text{span}}} \{z^n \clq_v: n \in \Z_+\} = H^2 \text{ or } \{0\}.
\]
If the above is equal to $H^2$, then $u H^2=H^2$, a contradiction to the fact that $u$ is a nonconstant inner function. It follows that
\[
\clq_v = \clq = \{0\},
\]
which implies $uH^2 \subseteq \ker H_\vp$. This completes the proof of the forward direction. The reverse direction is straightforward.

\noindent For (2) $\Rightarrow$ (3), assume that $uH^2 \subseteq \ker H_\vp$. In particular, $H_{\varphi} u = 0$. From the definition of $H_\varphi$, we find
\[
H_{\varphi}u = P_+ \varphi Ju=0.
\]
This readily implies that
\[
\varphi Ju \in \overline{zH^2}.
\]
Equivalently
\[
\langle\varphi Ju,f\rangle_{L^2(\T)} =0,
\]
for all $f\in H^2$, and hence
\[
\langle P_+\varphi,\overline{Ju}f\rangle_{H^2} = 0,
\]
for all $f\in H^2$. Therefore
\[
P_+\varphi \in \clq_{\overline{Ju}}.
\]
This proves that (2) $\Rightarrow$ (3). Finally, assume that the above property holds. Then, as above, $\langle \varphi Ju, f\rangle_{L^2(\T)} = 0$ for all $f \in H^2$, which implies $\varphi Ju \in \overline{zH^2}$. Since $\vp J u \in L^\infty(\T)$, we conclude that
\[
\varphi Ju \in \overline{zH^{\infty}(\mathbb{D})}.
\]
This gives
\[
H_{\varphi}uf = P_+\varphi((Ju)(Jf))=0,
\]
for $f\in H^2$, and hence $u H^2\subseteq \ker H_{\varphi}$. This completes the proof of Theorem \ref{invariant}. \qed

\bigskip

In particular, if $\ker H_\vp = \{0\}$, then except for $H^2$, none of the Beurling-type subspaces can be invariant under $H_\vp$. We illustrate this with a common example. Consider the classical Hilbert's Hankel matrix $\Gamma$ defined by \cite[page 191]{Young}
\[
\Gamma = \begin{bmatrix}
1 & \frac{1}{2} & \frac{1}{3} & \cdots
\\
\frac{1}{2} & \frac{1}{3} & \frac{1}{4} & \cdots
\\
\frac{1}{3} & \frac{1}{4} & \frac{1}{5} & \cdots
\\
\cdot & \cdot & \cdot & \cdots
\end{bmatrix}.
\]
Let us denote $H_\psi$ the Hankel operator on $H^2$ corresponding to Hilbert's Hankel matrix $\Gamma$, where $\psi$ is a symbol defined by
\[
\psi(e^{it})=ie^{-it}(\pi-t),
\]
for $t\in \T$. It is well-known that $H_\psi$ has trivial kernel \cite{Mag}. As a result, we have the following corollary concerning the lattice of the invariant subspaces of Hilbert's Hankel matrix:

\begin{corollary}\label{cor Hilbert matrix}
Let $H_\psi$ denote the Hankel operator on $H^2$ corresponding to Hilbert's Hankel matrix $\Gamma$. Then
\[
\{u H^2: u \in H^\infty(\D) \text{ inner}\} \cap \text{Lat} H_\psi = \{H^2\}.
\]
\end{corollary}

This result stands in contrast to composition operators, since there is always a Beurling-type subspace that is invariant under a given composition operator (see Matache \cite{Matache}, and also see \cite{Bose}).

Now we turn to model spaces that reduce Hankel operators. We need to recall some basics about inner function arithmetic. Note that an inner function $u \in H^\infty(\D)$ is said to be the \textit{greatest common divisor} of a pair of inner functions $u_1, u_2 \in H^\infty(\D)$ if
\begin{enumerate}
\item  $u$ divides both $u_1$ and $u_2$, and
\item if an inner function $u_3 \in H^\infty(\D)$ divides both $u_1$ and $u_2$, then $u_3$ also divides $u$.
\end{enumerate}
In the above case, we simply write
\[
u = \text{gcd}\{u_1, u_2\}.
\]
Furthermore, when we say that an inner function $u \in H^\infty(\D)$ divides another inner function $v \in H^\infty(\D)$, we mean that there is a function $w \in H^\infty(\D)$ (which will be forced to be an inner function anyway) such that
\[
v = wu.
\]

Consider a pair of inner functions $u_1, u_2 \in H^\infty(\D)$. Among the model spaces $\clq_{u_1}$ and $\clq_{u_2}$, the following identity exists:
\[
(\text{gcd}\{u_1, u_2\} H^2)^\perp = \clq_{\text{gcd}\{u_1, u_2\}} = \clq_{u_1} \cap \clq_{u_2}.
\]
Indeed, $f \in (\text{gcd}\{u_1, u_2\} H^2)^\perp$ if and only if
\[
\langle f, \text{gcd}\{u_1, u_2\} g \rangle_{H^2} = 0,
\]
for all $g \in H^2$. This is now equivalent to (in view of the definition of gcd)
\[
\langle f, u_1 g \rangle_{H^2} = 0 = \langle f, u_2 g \rangle_{H^2},
\]
for all $g \in H^2$. Equivalently, we have
\[
f \in \clq_{u_1} \cap \clq_{u_2},
\]
which completes the proof of the well-known claim.

Now we are ready for the complete classification of model spaces that reduce Hankel operators.

\begin{corollary}\label{reducing}
Let $u \in H^\infty(\D)$ be a nonconstant inner function, and let $\vp \in L^\infty(\T)$. Then, the following are equivalent:
\begin{enumerate}
\item $u H^2$ reduces $H_\varphi$.
\item $u H^2 \subseteq \ker H_\varphi \cap \ker H^*_\varphi$.
\item $P_+\varphi\in \big(\text{gcd} \{u, \overline{Ju}\} H^2\big)^\perp$.
\end{enumerate}
\end{corollary}
\begin{proof}
If $u H^2$ reduces $H_\varphi$, then, in particular, $u H^2$ remains invariant under both $H_\vp$ and $H_\vp^*$. Since $H_\vp^*$ is also a Hankel operator, (2) simply follows from Theorem \ref{invariant}. The reverse direction (2) $\Rightarrow$ (1) is easy: $u H^2 \subseteq \ker H_\varphi \cap \ker H^*_\varphi$ readily implies, as a general fact, that $u H^2$ (and hence $\clq_{u}$) reduces $H_\varphi$. This proves that (1) and (2) are equivalent.

\noindent For (2) $\Rightarrow$ (3), assume that $u H^2 \subseteq \ker H_\varphi \cap \ker H^*_\varphi$. Then $u H^2 \subseteq \ker H_\varphi$, together with Theorem \ref{invariant}, gives that
\[
P_+\varphi \in \clq_{\overline{Ju}}.
\]
On the other hand, $u H^2 \subseteq \ker H_{\varphi}^*$ implies $H_{\varphi}^* u = 0$. In view of $H_{\varphi}^* = H_{\overline{J\varphi}}$ \cite[Page 432]{Peller}, we have $H_{\overline{J\varphi}} u = 0$, and hence
\[
P_+\overline{J\varphi}Ju=0.
\]
This implies $\overline{J\varphi}Ju\in\overline{zH^2}$. Since $\overline{J\varphi}Ju \in L^\infty(\T)$, it readily follows that
\[
\overline{J\varphi}Ju\in\overline{zH^{\infty}(\mathbb{D})}.
\]
There exists $h\in \overline{H^{\infty}(\mathbb{D})}$ such that
\[
\overline{J\varphi}Ju=\bar{z}h.
\]
We take the conjugate of each side first, and then apply $J$ to either side to get
\[
\varphi\bar{u}=\bar{z}\overline{Jh}\in\overline{zH^{\infty}(\mathbb{D})}.
\]
Since $\varphi\bar{u}\in\overline{zH^{\infty}(\D)}$, for all $f\in H^2$, we have $\langle\varphi\bar{u}, f \rangle_{L^2(\T)} = 0$. Equivalently
\[
\langle P_+\varphi, u f \rangle_{H^2} = 0.
\]
This implies $P_+\varphi \in \clq_u$, and hence
\[
P_+ \varphi \in \clq_u \cap \clq_{\overline{Ju}} = \big(\gcd\{u,\overline{Ju}\}H^2\big)^{\perp}.
\]
Finally, for (3) $\Rightarrow$ (2), assume that $P_+ \varphi \in \big(\gcd\{u,\overline{Ju}\}H^2\big)^{\perp} = \clq_u \cap \clq_{\overline{Ju}}$. Theorem \ref{invariant} and the fact that $H_{\varphi}^* = H_{\overline{J\varphi}}$ immediately imply that $u H^2 \subseteq \ker H_\varphi \cap \ker H_\vp^*$.
\end{proof}

In the context of Theorem \ref{invariant}, it is worth mentioning that the condition $u H^2 \subseteq \ker H_\vp$ can be expressed equivalently as
\[
H_\vp T_u = 0.
\]
Within the same framework, Douglas' range inclusion theorem states that $u H^2$ is invariant under $H_\vp$ if and only if
\[
H_\vp T_u = T_u Y,
\]
for some $Y \in \clb(H^2)$. It is also curious to note that in this case, $H_\vp T_u$ is a Hankel operator. Similarly, Corollary \ref{reducing} implies that $u H^2$ reduces $H_\vp$ if and only if $H_\vp T_u = 0$ and $H_\vp^* T_u = 0$.

We conclude this section with two more invariant subspace results. Recall that Theorem \ref{invariant} provides a complete classification of Beurling-type subspaces invariant under Hankel operators. Using the same result, we can also classify model spaces invariant under Hankel operators. Indeed, recall the identity \cite[Page 432]{Peller}
\[
H_{\varphi}^* = H_{\overline{J\varphi}}
\]
for all $\vp \in L^\infty(\T)$, used in the previous corollary. Let $u \in H^\infty(\D)$ be a nonconstant inner function. Given that $uH^2$ is invariant under $H_\vp$ if and only if $\clq_u$ is invariant under $H_\vp^*$, the following is immediately implied by the above adjoint formula and Theorem \ref{invariant}:

\begin{corollary}\label{model invariant}
Let $u \in H^\infty(\D)$ be a nonconstant inner function, and let $\vp \in L^\infty(\T)$. Then, the following are equivalent:
\begin{enumerate}
\item $\clq_u$ is invariant under $H_\varphi$.
\item $u H^2 \subseteq \ker H_{\overline{J\varphi}}$.
\item $P_+\varphi \in \clq_{u}$.
\end{enumerate}
\end{corollary}
 
 

Finally, we return to the Hankel operator $H_\psi$ on $H^2$ corresponding to Hilbert's Hankel matrix $\Gamma$. Note that $H_\psi$ is a self-adjoint operator, and consequently, Corollary \ref{cor Hilbert matrix} implies:

\begin{corollary}\label{model cor Hilbert matrix}
Let $H_\psi$ denote the Hankel operator on $H^2$ corresponding to Hilbert's Hankel matrix $\Gamma$. Then
\[
\{\clq_u : u \in H^\infty(\D) \text{ inner}\} \cap \text{Lat} H_\psi = \{\{0\}\}.
\]
\end{corollary}

In the above, one needs to use the fact that the model spaces are always proper subspaces of $H^2$.
 

\section{Revisiting liftings and examples}\label{sec: revisit lift}

Let us now go back to the lifting problem in Section \ref{sec: lifting thm}. At that point, the primary question that remained unanswered was the existence of intertwiners between a model operator and its adjoint. The representation of the lift was an additional concern.

We recall from Theorem \ref{thm: lift 1} that if the given intertwiner $X$ on the model space $\clq_u$ admits a lift to $H_\vp$, then $\clq_u$ reduces $H_\vp$. Given Corollary \ref{reducing}, which classifies model spaces that reduce Hankel operators, we can now relate the model space's inner functions to the Hankel operator symbols. In other words, in the following theorem, we directly relate $u$ and $\vp$:

\begin{theorem}\label{X phi}
Let $u \in H^\infty(\D)$ be an inner function. Then there exists $X \in \clb(\clq_u)$ such that
\[
S_u^*X=XS_u,
\]
if and only if there exists $\varphi\in L^{\infty}(\mathbb{T})$ such that
\[
P_+\varphi \in \big(\gcd\{u,\overline{Ju}\}H^2\big)^{\perp}.
\]
Moreover, in this case, $\clq_u$ reduces $H_\vp$, and $X=H_\varphi|_{\clq_u}$.
\end{theorem}
\begin{proof}
Suppose $S_u^*X=XS_u$ for some $X \in \clb(\clq_u)$. By Theorem \ref{thm: lift 1}, there exists $\varphi\in L^{\infty}(\mathbb{T})$ such that
\[
H_{\varphi}=\begin{bmatrix}
X&0
\\
0&0
\end{bmatrix}.
\]
In particular, $\clq_u^\perp \subseteq \ker H_\varphi \cap \ker H^*_\varphi$. Corollary \ref{reducing} then implies that
\[
P_+\varphi \in \big(\gcd\{u,\overline{Ju}\}H^2\big)^{\perp}.
\]
It is also evident that $X=H_\varphi|_{\clq_u}$. For the converse direction, by Corollary \ref{reducing} we again have that $\clq_u^\perp \subseteq \ker H_\varphi \cap \ker H^*_\varphi$. Thus the matrix representation of $H_\varphi$ with respect to the decomposition $H^2 = \clq_u \oplus \clq_u^\perp$ is
\[
H_\vp = \begin{bmatrix}
H_\varphi|_{\clq_u} & 0
\\
0 & 0
\end{bmatrix}.
\]
Then the rest of the proof follows from Theorem \ref{thm: lift 1} with $X=H_\varphi|_{\clq_u}$.
\end{proof}

We now shift our focus to the issue of finding a nonzero solution to the Hankel-type intertwiner problem on model spaces. More specifically, in the following, we present the proof of Theorem \ref{thm: nonzero lift}.

\bigskip

\noindent \textsf{Proof of Theorem \ref{thm: nonzero lift}:} Suppose $S_u^*X=XS_u$ for some nonzero $X\in\clb(\clq_u)$. By Theorem \ref{X phi}, there exists $\varphi\in L^\infty(\mathbb{T})$ such that $X=H_\varphi|_{\clq_u}$ and $P_+\varphi\in(\gcd\{u, \overline{Ju}\}H^2)^\perp$.
If $\gcd\{u, \overline{Ju}\} = 1$, then $(\gcd\{u, \overline{Ju}\}H^2)^\perp = \{0\}$ and hence
\[
P_+\varphi = 0.
\]
Then $H_\vp = 0$ and consequently $X=H_\varphi|_{\clq_u}=0$ - a contradiction to our assumption that $X\ne 0$. Conversely, suppose
\[
\theta:= \gcd\{u, \overline{Ju}\} \neq 1.
\]
Define
\begin{equation}\label{eqn: nontrivial phi}
\varphi: = T_z^*\theta.
\end{equation}
Since $\ker T_z^* = \mathbb{C}$, it follows that $\varphi\neq 0$. Moreover, by the definition of $\theta$, we have $T_z^* \theta \in \clq_u$ and also $T_z^* \theta \in \clq_{\overline{J u}}$, and consequently
\[
T_z^* \theta \in (\gcd\{u, \overline{Ju}\}H^2)^\perp.
\]
Since $T_z^* \theta = \frac{\theta - \theta(0)}{z}$, for each $w\in\mathbb{T}$, we have
\[
|\varphi(w)| = \Big|\frac{\theta(w)-\theta(0)}{w}\Big| \le |\theta(w)|+|\theta(0)|.
\]
Therefore
\[
|\varphi(w)| \leq 1+|\theta(0)|,
\]
a.e. on $\mathbb{T}$, which implies $\varphi\in L^\infty(\mathbb{T})$. Then, from $\varphi \in H^2$, we conclude that
\[
P_+\varphi = \varphi = T_z^*\theta \in (\gcd\{u, \overline{Ju}\} H^2)^\perp.
\]
Define
\[
X:=H_\varphi|_{\clq_u}.
\]
By Theorem \ref{X phi}, $S_u^*X=XS_u$, and hence it only remains to prove that $X\ne 0$. In view of $H^2=uH^2\oplus \clq_u$, write
\[
1=uh+g,
\]
for some $h \in H^2$ and $g \in \clq_u$. By Corollary \ref{reducing}, in particular, we have $uH^2\subseteq \ker H_\varphi$. Then
\[
X g = H_\varphi 1 = P_+\varphi = \varphi\neq 0,
\]
proves that $X \neq 0$. This completes the proof of Theorem \ref{thm: nonzero lift}. \qed

\bigskip

It is important to observe, in view of \eqref{eqn: nontrivial phi}, that $X:= H_\vp|_{\clq_u}$ is a nonzero solution to the operator equation $S_u^*X=XS_u$, where
\[
\vp = T_z^* \gcd\{u, \overline{Ju}\}.
\]
Clearly, there are alternative solutions $X$ to the operator equation $S_u^*X=XS_u$. Suppose
\[
\theta = \gcd\{u, \overline{Ju}\} \neq 1.
\]
Assume that $\dim \clq_\theta = 1$. Then, all the nonzero intetwiners $X$ are given by
\[
X = H_{\alpha T_z^* \theta}|_{\clq_u},
\]
where $\alpha \in \mathbb{C}$. Next, assume that
\[
\dim \clq_\theta >1.
\]
For $j=1,2$, set
\[
\vp_j = T_z^{*j} \theta,
\]
and define
\[
X_j = H_{\vp_j}|_{\clq_u}.
\]
We claim that $X_1 \neq X_2$ are nonzero solutions to $S_u^*X=XS_u$. To see this, as in the proof of the above theorem, we write
\[
1=uh+g,
\]
for some $h \in H^2$ and $g \in \clq_u$. It is now easy to see that
\[
X_j g = H_{\vp_j} 1 = P_+ \vp_j = \vp_j,
\]
for all $j=1,2$. Since $\vp_1$ and $\vp_2$ are nonzero distinct functions, we conclude that $X_1$ and $X_2$ are nonzero distinct intertwiners.

Now we focus on model spaces corresponding to Blaschke products. Given each $\alpha \in \D$, define the \textit{Blaschke factor} $b_\alpha$ by
\[
b_{\alpha}(z) = \frac{\alpha - z}{1-\bar{\alpha} z} \qquad (z \in \D).
\]
Let $\Lambda \subseteq \mathbb{N}$ be a finite or countably infinite set. Pick a set of complex numbers $\{\alpha_n\}_{n \in \Lambda} \subseteq \D$ (repetition is allowed). Assume the Blaschke condition that
\[
\sum_{n \in \Lambda} (1 - |\alpha_n|) < \infty.
\]
Consider the Blaschke product
\[
u = \prod_{n \in \Lambda} b_{\alpha_n}.
\]
Then
\[
\overline{J u} = \prod_{n \in \Lambda} b_{\bar{\alpha}_n},
\]
and consequently $\gcd\{u, \overline{Ju}\}\ne 1$ if and only if there exist $p, q \in \Lambda$ such that $\bar{\alpha}_p = \alpha_q$. This proves the following corollary to the above theorem:

\begin{corollary}\label{roots}
Given a Blaschke product $u = \displaystyle \prod_{n \in \Lambda} b_{\alpha_n}$ as above, there exists a nonzero $X \in \clb(\clq_u)$ such that $S_u^*X=XS_u$ if and only if there exist $p, q \in \Lambda$ such that
\[
\bar{\alpha}_p = \alpha_q.
\]
\end{corollary}

We remark that if $\Lambda$ is a finite set, then $u$ becomes a finite Blaschke product that corresponds to finite-dimensional model spaces.
 

It is now clear that one can construct concrete examples of model spaces that admit or do not admit nonzero intertwiners. We present two extreme examples of this claim. First, consider the inner function
\[
u(z) = b_{\frac{i}{2}}(z) \exp\Big(\frac{z+1}{z-1}\Big),
\]
for all $z \in \D$. Note that second factor is a singular inner function. In this case, it is evident that
\[
\gcd\{u, \overline{Ju}\} = \exp\Big(\frac{z+1}{z-1}\Big) \neq 1.
\]
Therefore, $X:= H_\vp|_{\clq_u}$ on $\clq_u$ yields a nonzero solution to $S_u^*X=XS_u$, where
\[
\vp = T_z^* \exp\Big(\frac{z+1}{z-1}\Big).
\]
On the other hand, if $u = b_{\frac{i}{2}}$, then $\gcd\{u, \overline{Ju}\} = 1$, and consequently there is no nonzero solution to the equation $S_u^*X=XS_u$.

Next, we shift our focus to another question:

\begin{question}
Fix a symbol $\vp \in L^\infty(\T)$. Is there a model space $\clq_u$ and map $X \in \clb(\clq_u)$ such that $S_u^* X = X S_u$ lifts to $H_\vp$?
\end{question}

To answer this, we retrieve the Hankel operator on $H^2$ once more, which corresponds to Hilbert's Hankel matrix as stated in Corollary \ref{cor Hilbert matrix}.

\begin{example}\label{exam: Hilbert matrix}
Let $\clq_u$ be a nontrivial model space. Denote by $\psi \in L^\infty(\mathbb{T})$ the symbol that corresponds to Hilbert's Hankel matrix. We know from Corollary \ref{cor Hilbert matrix} that the model spaces do not reduce $H_\psi$. This implies that if $S_u^* X = X S_u$ for some $X \in \clb(\clq_u)$ and $X = H_\vp|_{\clq_u}$ for some $\vp \in L^\infty(\T)$, then
\[
H_\vp \neq H_\psi.
\]
\end{example}

As a result, none of the intertwiners between a model operator and its adjoint admits a lift to Hilbert's Hankel matrix.

\section{Concluding Remarks}\label{sec: remark}

The goal of this final section is to outline potential directions for refining the results obtained so far in this paper, as well as to demonstrate the triviality of further improvement. We will also talk about inner functions connected to the kernels of Hankel operators and how they relate to the Beurling-type invariant subspaces of Hankel operators.

\subsection{Toeplitz formulation} The formulation of the lifting problem on model spaces presented in this paper can be applied to a wide range of problems. Although interesting to study (and have been in many cases), similar questions in other contexts may not always have pleasing answers. In the following, we discuss one such situation that becomes trivial. Recall that $T_\vp$ denotes the Toeplitz operator with symbol $\vp \in L^\infty(\T)$, where
\[
T_\vp = P_{+} L_\varphi|_{H^2}.
\]

Recall the algebraic classification of Toeplitz operators (see Brown and Halmos \cite{BH}): $X\in \clb(H^2)$ is a Toeplitz operator if and only if $T_z^* X T_z = X$. We are specifically interested in the model space variant of this classification. 

\begin{proposition}\label{toeplitz}
Let $\clq_u \subseteq H^2$ be a model space, and let $X\in\clb(\clq_u)$. Then
\[
S_u^*XS_u=X,
\]
if and only if
\[
X=0.
\]
\end{proposition}
\begin{proof}
We follow the notations from Section \ref{sec: lifting thm}. Therefore, $S_u = i_{\clq_u}^* T_z i_{\clq_u}$, and hence $S_u^* X S_u = X$ implies
\[
i_{\clq_u}^* T_z^* i_{\clq_u} X i_{\clq_u}^* T_z i_{\clq_u} = X,
\]
which gives
\[
(P_{\clq_u} T_z^* i_{\clq_u}) X (i_{\clq_u}^* T_z P_{\clq_u}) = i_{\clq_u} X i_{\clq_u}^*.
\]
Using the identity $P_{\clq_u} T_z^* i_{\clq_u} = T_z^* i_{\clq_u}$ from \eqref{eqn: PQ Tz}, we see that
\[
T_z^* (i_{\clq_u} X i_{\clq_u}^*) T_z = i_{\clq_u} X i_{\clq_u}^*,
\]
and then the Brown and Halmos classification of Toeplitz operators implies $i_{\clq_u} X i_{\clq_u}^* = T_\vp$ for some $\vp \in L^\infty(\T)$. Finally, recall from \eqref{eqn: iQ X} that
\[
i_{\clq_u} X i_{\clq_u}^* = \begin{bmatrix}
X&0\\
0&0
\end{bmatrix},
\]
and subsequently
\[
T_\vp = \begin{bmatrix}
X&0\\
0&0
\end{bmatrix}.
\]
In particular, $\clq_u^\perp \subseteq \ker T_\varphi \cap \ker T_\varphi^*$. On the other hand, by Coburn's lemma on Toeplitz operators \cite[Theorem 4.1]{Coburn}, if $\vp \neq 0$, then we have
\[
\ker T_\varphi \cap \ker T_\varphi^* = \{0\},
\]
yielding that $\clq_u^\perp = \{0\}$. This contradicts the fact that $\clq_u^\perp$, which is a Beurling-type subspace, cannot be trivial. Therefore, $T_\vp = 0$ and hence $X = 0$. This completes the proof.
\end{proof}

Therefore, the Toeplitz operator problem in the setting of model spaces has only a trivial solution.

\subsection{Dilation formulation} The aim here is to discuss some possible generalizations of the results of this paper with the aid of dilation theory. While it may not be entirely evident if all the results accept higher generalizations, like in the context of vector-valued Hankel operators, some of the results admit immediate modifications. For example, we can easily convert Theorem \ref{thm: lift 1} into the terminology of isometric dilation (see \cite{Foias} for all the terms).

Let $T$ acting on a Hilbert space $\clh$ be a contraction, and let $V$ on a Hilbert space $\clk (\supseteq \clh)$ be the minimal isometric dilation of $T$. A bounded linear operator $X: \clk \raro \clk$ is said to be $V$-Hankel if
\[
V^* X = X V.
\]
With a little more work, the same line of argument will lead to the general statement of Theorem \ref{thm: lift 1}, as follows:

\begin{proposition}
Let $X: \clh \raro \clh$ be a bounded linear operator. Then $T^* X=X T$ if and only if there exists a $V$-Hankel operator $Y$ on $\clk$ such that
\[
Y=\begin{bmatrix} X &0\\
0& 0\end{bmatrix},
\]
with respect to the decomposition $\clk=\clh\oplus (\clk\ominus \clh)$. In particular, in the above case, we have the following:
\[
X= Y|_{\clh}.
\]
\end{proposition}

It is, however, unclear how to obtain a concrete representation of the $V$-Hankel operator $Y$.

\subsection{Hankel kernels} Let $\vp \in L^\infty(\T)$, and suppose that $\ker H_\vp \neq \{0\}$. We know by the Hankel-condition \eqref{eqn: Hankel class}, there exists an inner function $w_\vp \in H^\infty(\D)$ such that $\ker H_\vp$ is a nonzero shift invariant subspace of $H^2$, namely $w_\vp H^2$. That is
\[
\ker H_\vp = w_\vp H^2.
\]
Consider a nonconstant inner function $u \in H^\infty(\D)$. We recall from Theorem \ref{invariant} that $u H^2$ is invariant under $H_\varphi$ if and only if $u H^2 \subseteq \ker H_\varphi$. Therefore, we are in the following situation:
\[
u H^2 \subseteq \ker H_\varphi = w_\vp H^2.
\]
It is now evident that the above condition is equivalent to the existence of an inner function $v \in H^\infty(\D)$ such that
\[
u = w_\vp v.
\]
This proves the following characterization of Beurling-type invariant subspaces of Hankel operators:

\begin{theorem}
Let $\vp \in L^\infty(\T)$, and suppose
\[
\ker H_\vp = w_\vp H^2,
\]
for some inner function $w_\vp \in H^\infty(\D)$. If $u \in H^\infty(\D)$ is an inner function, then $u H^2$ is invariant under $H_\varphi$ if and only if there exists an inner function
$v \in H^\infty(\D)$ such that
\[
u = w_\vp v.
\]
\end{theorem}

In other words, all inner multiples of $w_\vp$ implement Beurling-type invariant subspaces of $H_\vp$, where $w_\vp$ is the Beurling inner function of $\ker H_\vp$ (provided the kernel space is nonzero).

The (well-known) clarification above has shown us that Hankel operators yield inner functions through their kernel spaces (as long as they are nonzero). Let us complete this (well-known) link of inner functions with the kernels of Hankel operators before we call the paper to an end. For any inner function $u \in H^\infty(\D)$, there is always a symbol $\vp \in L^\infty(\T)$ such that
\[
\ker H_\vp = u H^2.
\]
This fact can be found in \cite[page 15, Theorem 2.4]{Peller}, particularly in the setting of Hankel operators from $H^2$ to $(H^2)^\perp$. For the current setting, simply define the symbol
\[
\vp = \bar{z} \overline{J u}.
\]
In this case, one sees that
\[
\ker H_{\bar{z} \overline{J u}} = u H^2.
\]
We present an outline of the proof, as it is simplified in this setting. First, we note that $\overline{(J u)(z)} (Ju)(z) = 1$ for a.e. $z \in \T$. For all $f \in H^2$, we have
\[
\begin{split}
H_{\bar{z} \overline{J u}} uf & = P_+ (\bar{z} \overline{J u} Ju Jf)
\\
& = P_+(\bar{z} Jf)
\\
& = 0,
\end{split}
\]
and hence $uH^2 \subseteq \ker H_{\bar{z} \overline{J u}}$. For the reverse inclusion, we now appeal to Beurling to get an inner function $v \in H^\infty(\D)$ such that
\[
\ker H_{\bar{z} \overline{J u}} = v H^2.
\]
We already know that $uH^2 \subseteq v H^2$. Now $v = v \times 1 \in v H^2$ implies that $H_{\bar{z} \overline{J u}} v = 0$, and hence
\[
P_+ (\bar{z} \overline{J u}) J v = 0.
\]
Again, in view of $uH^2 \subseteq v H^2$, there exists an inner function $w \in H^\infty(\D)$ such that
\[
u = v w.
\]
This implies $\overline{Ju} = \overline{Jv}\; \overline{Jw}$, and then the above identity yields
\[
P_+ (\bar{z}  \overline{Jw}) = 0.
\]
On the other hand, writing
\[
w = \sum_{n=0}^{\infty} \alpha_n z^n,
\]
we find, on $\T$, that
\[
\bar{z} \overline{Jw} = \bar{\alpha}_0 \bar{z} + \sum_{n=0}^{\infty} \bar{\alpha}_{n+1}z^n.
\]
Consequently, $P_+ (\bar{z}  \overline{Jw}) = 0$ forces that
\[
\alpha_n = 0,
\]
for all $n \geq 1$, and hence, $w(z) = \alpha_0$ for all $z \in \D$. This shows that
\[
u = \alpha_0 v,
\]
and $\alpha_0 \in \T$. This proves that $\ker H_{\bar{z} \overline{J u}} = u H^2$.

As presented in this paper, the principle of lifting and the existence of Beurling-type invariant and reducing subspaces rely heavily on the kernel spaces of Hankel operators. The relationship between inner functions and the nontrivial kernel spaces of Hankel operators suggests that the theory described here encompasses a wide range of functions, spaces, and operators.

\vspace{0.1in} \noindent\textbf{Conflict of interest:} The author states that there is no conflict of interest. No data sets were generated or analyzed during the current study.

\vspace{0.1in} \noindent\textbf{Acknowledgement:} We thank the referees for their constructive feedback, which also helped in improving the presentation of the paper. The research of the first named author is supported by NBHM (National Board of Higher Mathematics, India) Ph.D. fellowship No. 0203/13(47)/2021-R\&D-II/13177. The research of the third named author is supported by NBHM (National Board of Higher Mathematics, India) Post-doctoral fellowship No. 0204/27/(26)/2023/R\&D-II/11926. The research of the fourth named author is supported in part by TARE (TAR/2022/000063) by SERB, Department of Science \& Technology (DST), Government of India.

\bibliographystyle{amsplain}

\end{document}